\documentclass[11pt,a4paper,english,reqno]{amsart}
\usepackage[a4paper,footskip=1.5em]{geometry}
\usepackage{amsmath,amssymb,amsthm,mathtools}

\usepackage{hyperref}
\hypersetup{colorlinks=true,linkcolor=blue,citecolor=blue,pdfpagemode=UseNone,pdfstartview={XYZ null null 1.00}}

\theoremstyle{plain}
\newtheorem*{theorem*}{Theorem}
\newtheorem{theorem}{Theorem}[section]
\newtheorem{lemma}[theorem]{Lemma}

\newtheorem{proposition}[theorem]{Proposition}
\newtheorem*{claim*}{Claim}

\newtheorem{conjecture}[theorem]{Conjecture}

\theoremstyle{remark}

\let\originalleft\left
\let\originalright\right
\renewcommand{\left}{\mathopen{}\mathclose\bgroup\originalleft}
\renewcommand{\right}{\aftergroup\egroup\originalright}

\begin{document}

\title{A note on some conjectures about combinatorial models for RNA secondary structures}

\author{Adam Zsolt Wagner}
\thanks{Department of Mathematics, ETH, Z\"urich, Switzerland. Email:
\href{mailto:zsolt.wagner@math.ethz.ch} {\nolinkurl{zsolt.wagner@math.ethz.ch}}.\\
Research has been performed while the author was a PhD student at the University of Illinois at Urbana-Champaign}

\maketitle

\begin{abstract}

We resolve two conjectures of Black-Drellich-Tymoczko about the numbers of valid plane trees for given primary sequences.

\end{abstract}

\section{Introduction}

Recall some definitions from \cite{black,heitsch}: a set $A$ is a {\it complementary alphabet of size} $m$ if it consists of $m$ pairs of complementary letters $\{A,\bar{A}, B, \bar{B}, \ldots\}$. We refer to $\bar{A}$ as the \emph{complement} of $A$, and vice versa $A$ is the complement of $\bar{A}$. An example of such an alphabet is the set of the four nucleotides ${A,C,G,U}$ with the complementary pairs being the Watson--Crick complements $A-U$ and $C-G$.

Let $T$ be a rooted plane tree with $n$ edges and let $P =p_{1}p_{2}\ldots p_{2n}$ be a word in a complementary alphabet $\mathcal{A}$. Label the edges of $T$ by $P$ according to a counter-clockwise walk around the boundary of $T$, starting at the root, so that every edge receives two letters (one letter going 'down' and one letter going `up'). We say $T$ is $P$-valid if every edge in $T$ receives complementary letters by this method. As an example, consider the word $P=AA\bar{A}B\bar{B}\bar{A}\bar{B}B$.

\begin{center}
\includegraphics[width=65.70mm]{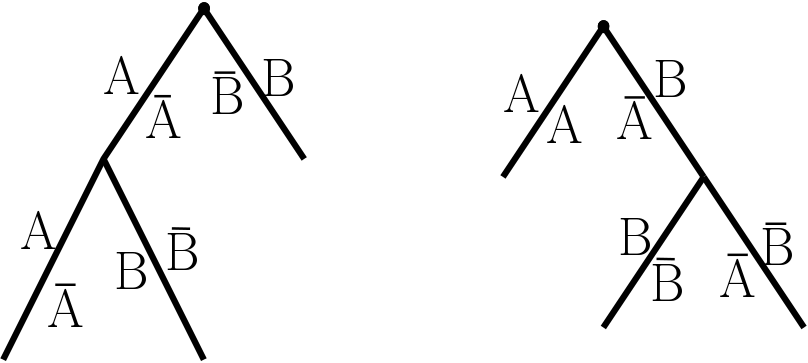}
\end{center}

The tree on the left in $P$-valid, as every edge received complementary pairs of letters. The tree on the right is not $P$-valid, as three out of its four edges do not receive complementary letters.

Let $\mathcal{P}(n, m)$ be the set of words $P$ of length $2n$ in a complementary alphabet $\mathcal{A}$ of size $m,$ with at least one $P$-valid plane tree. Let $V(P)$ be the set of $P$-valid plane trees for a fixed word $P$. Let $N(n,m,k)$ be the set of words $P\in \mathcal{P}(n, m)$ such that $|V(P)| =k$. Finally, let $R(n, m)$ be the set of $k\in \mathbb{N}^{+}$ such that $N(n, m, k)$ is non-empty.
The following two conjectures were stated in \cite{black}:

\begin{conjecture}\label{conj:1}
 $R(n, m)=R(n, 1)$ for all $m$. That is, the set $R(n, m)$ depends only on $m.$
\end{conjecture}

We will show that Conjecture 1.1 is false, by finding a counterexample for $n=7.$

\begin{conjecture}\label{conj:2}
{\it For all} $k\in \mathbb{N}^{+}$ {\it there exists} $n, m$ {\it such that} $k\in R(n,m)$ .
\end{conjecture} 

We will show that slightly more is true. Our main result is the following:

\begin{theorem}\label{thm:1}
{\it For all} $k\in \mathbb{N}^{+}$ {\it there exists} $n$ {\it such that} $k\in R(n,1)$ .
\end{theorem}

Finally, we give a short application of a classical combinatorial puzzle about gas stations on a circular race track, to show that $R(n,m)\subseteq R(n+1,m)$ for all $n,m$.

\section{Proof of main results}

We will find it easier to work with {\it noncrossing matchings} instead of rooted plane trees. Everything that follows can be translated into the language of rooted plane trees if necessary. Rooted plane trees and noncrossing perfect matchings are two of many counting problems in combinatorics whose solution is given by Catalan numbers.
Given $2n$ labelled points on a circle, a \emph{noncrossing perfect matching} is a perfect matching of the points, such that when drawing the edges of the matching as straight line segments inside the circle no two edges cross each other.

The standard bijection between rooted labelled plane trees and noncrossing perfect matching goes as follows. First, write the numbers $1,2,\ldots, 2n$ on equidistant points on the unit circle $x^2+y^2=1$, starting with the number $1$ at position $(0,1)$ and traversing the circle counterclockwise.
Given a rooted plane tree $T$ with $n$ edges, write the numbers $1, 2. \ldots, 2n$ on the edges according to a counter-clockwise walk around the boundary of $T$, starting at the root vertex, each edge receiving two numbers (of different parity). The pairs of integers corresponding to the same edge of $T$ will be the edges of our matching $M$. 

Recall that given a word $P\in \mathcal{P}(n, m)$, a $P$-valid plane tree $T$ is such that, when wrapping $P$ around $T$, each edge receives complementary letters. Hence, when translated into the language of matchings, a $P$-valid noncrossing perfect matching is such that, if we write the letters of $P$ in order along a circle, each edge in the matching receives complementary letters as its endpoints.

First we present our counterexample to Conjecture~\ref{conj:1}. Let $\mathcal{A} := \{A, \bar{A}, B, \bar{B}\}$, where the pairs $(A,\bar{A})$ and $(B,\bar{B})$ represent the complementary letters. We claim that the word $P = B\bar{B}\bar{A}AA\bar{A}\bar{A}AB\bar{B}A\bar{A}A\bar{A}$ has $|V(P)| = 11$. Indeed, there are two ways to match up the {\it B}s with the $\bar{B}$s. If we match the adjacent $B\bar{B}$ pairs we get 7 possible matchings, if we match the further away $B\bar{B}$ pairs we get $2\cdot 2=4$ possible matchings. But according to Figure 9 in~\cite{black}, we have $11\not\in R(7,1)$.
\begin{figure}[h]
    \centering
    \includegraphics[width=65.70mm,height=61.30mm]{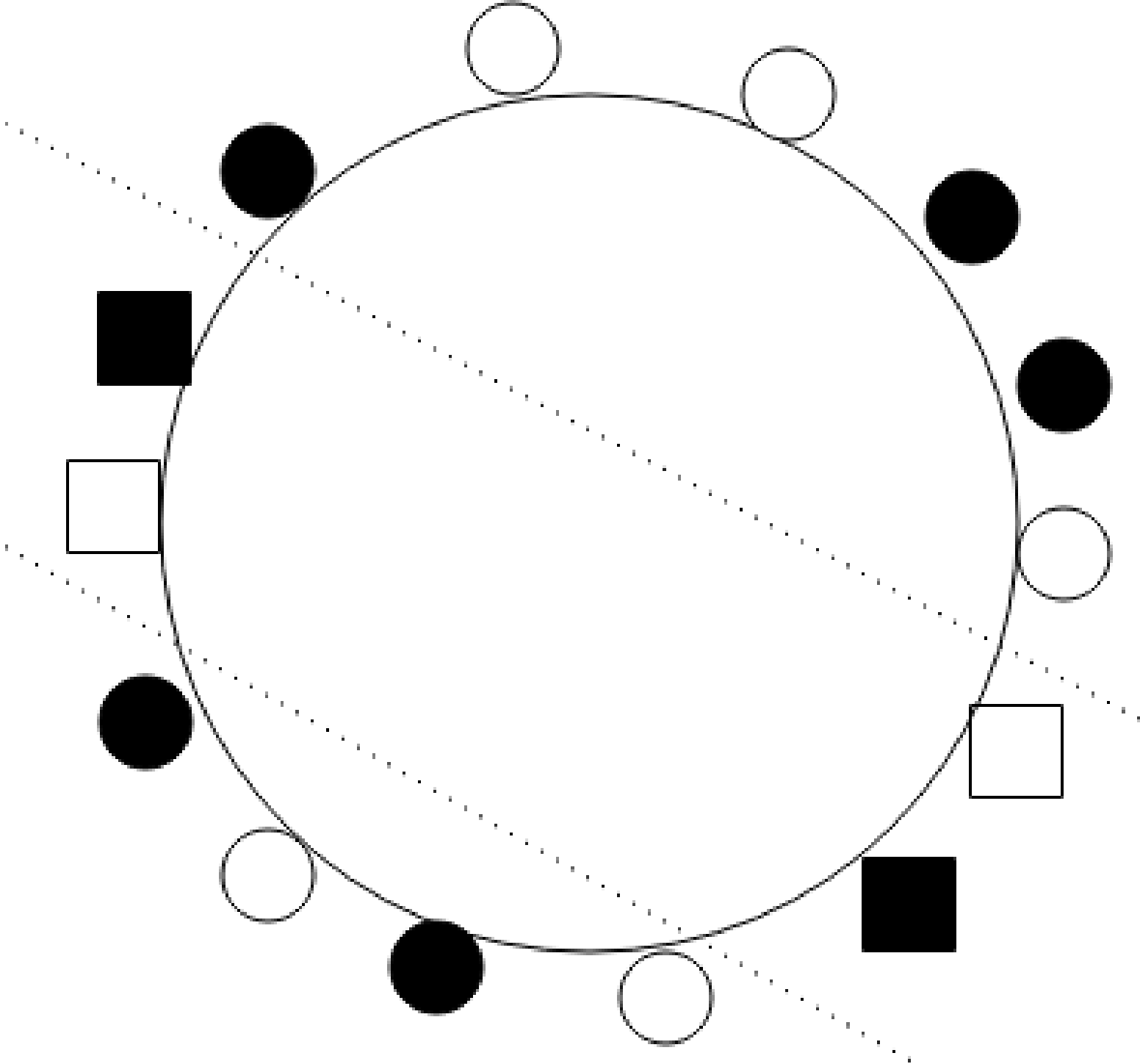}
    \caption{The counterexample to Conjecture~\ref{conj:1}. Circles are $A$s, squares are $B$s, black means complement (i.e. $\bar{A}$ or $\bar{B}$)}
    \label{fig:my_label}
\end{figure}

Now we turn our attention to proving Theorem~\ref{thm:1}. Define the word $P_{k,\ell}$ for each $k,\ell\in\mathbb{N}^+$ as follows.
$$ P_{k,\ell}=\overbrace{AAA\ldots A}^{k\text{ times}} ~  \overbrace{\bar{A}\bar{A}\bar{A}\ldots \bar{A}}^{k\text{ times}} ~ \overbrace{AAA\ldots A}^{\ell\text{ times}} ~  \overbrace{\bar{A}\bar{A}\bar{A}\ldots \bar{A}}^{\ell\text{ times}} = p_{1}p_{2}\ldots p_{2k+2l}.$$

Note that $|V(P_{1,1})| =2.$

\begin{lemma}
 {\it For} $k,  l\geq 1$ {\it we have}

$|V(P_{k,l})|= \left\{\begin{array}{ll}
|V(P_{k-1,l})| & if\ k>l\\
|V(P_{k,l-1})| & if\ l>k\\
1+|V(P_{k,l-1})| & if\ k=l
\end{array}\right.$
\end{lemma}
\begin{proof}
Suppose first that $k>l$. Suppose we match $p_{k}$ with $p_{i}$. What is the value of $i$? We cannot have $i<k$, as then $p_{i}$ and $p_{k}$ are not complementary letters. But we also cannot have $i>k+1$, because we need to have the same number of {\it A}s as $\bar{A}$s among $\{p_{k+1}\ldots p_{i-1}\}$ to get a noncrossing matching. So $i=k+1$. Hence $(p_{k}p_{k+1})$ is an edge of the matching and we may remove these two adjacent vertices to conclude the claim. The case $k<l$ is analogous to the above.
\begin{figure}[h]
    \centering
    \includegraphics[width=77.55mm,height=74.85mm]{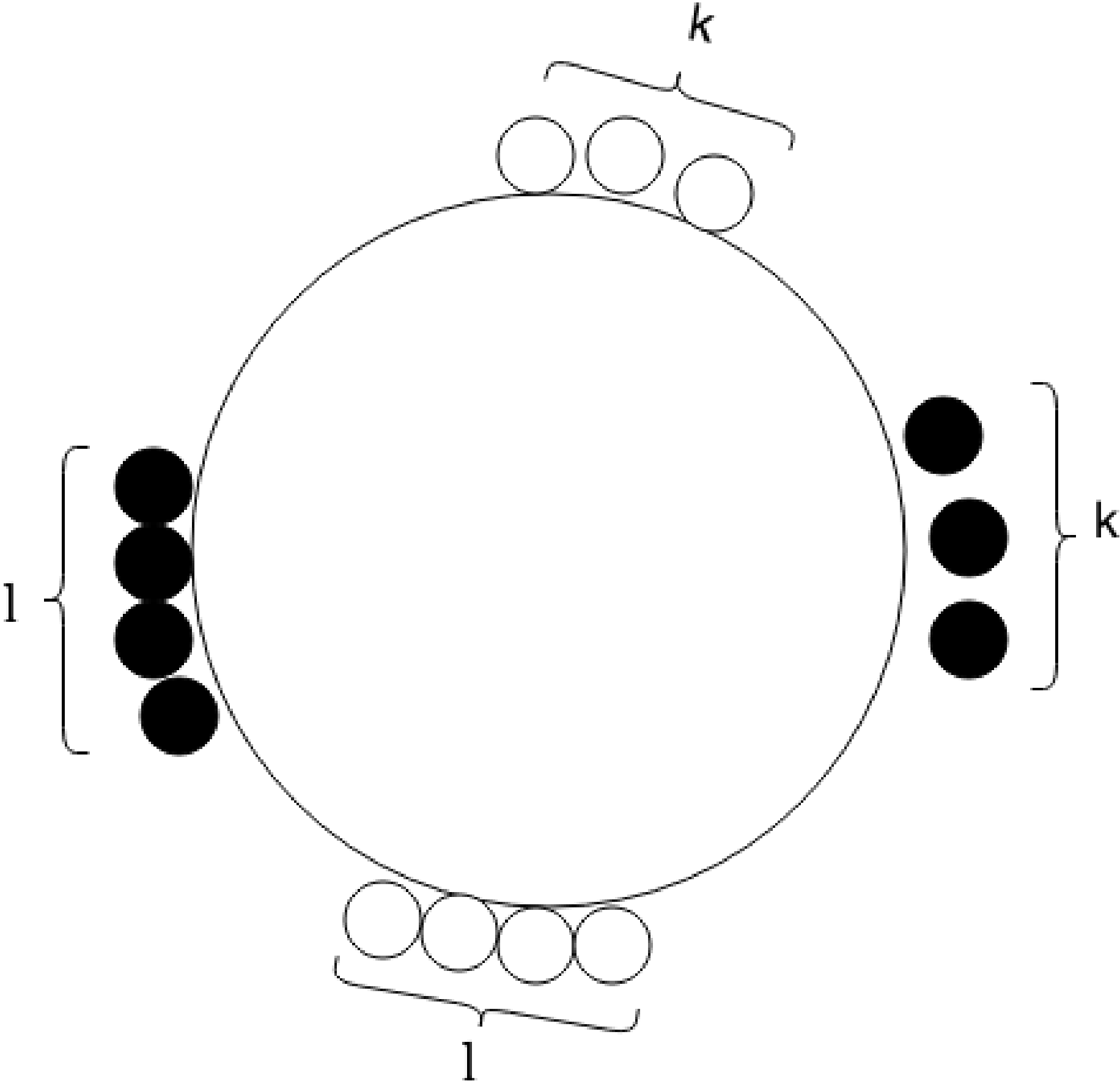}
    \caption{The word $P_{k,\ell}$ on a circle}
    \label{fig:my_label}
\end{figure}

Suppose now $k = l$. By the same logic as above, we either have $(p_{k}p_{k+1})$ as an edge (leaving us with $|V(P_{k,k-1})|$ possible matchings) or we have $(p_{k}p_{3k+1})$ as an edge, whence both sides of our edge have a unique noncrossing matching, as required.
\end{proof}

\begin{proof}[Proof of Theorem~\ref{thm:1}] Let $k \in \mathbb{N}^{+}$. Note that $|V(P_{k,k})| = 1 + |V(P_{k,k-1})| = 1 + |V(P_{k-1,k-1})|$, hence by induction we have $|V(P_{k,k})| = k+1$. So $k+1 \in R(2k, 1)$ for all $k.$ 
\end{proof}

We finish with an application of the following classical puzzle.
\begin{theorem}[\cite{boll, lovasz, winkler}\label{thm:gem}]
Along a speed track there are some gas stations. The total amount of gasoline available in them is equal to what our car (which has a very large tank) needs for going around the track. Then there is a gas station such that if we start there with empty tank, we shall be able to go around the track without running out of gasoline. 
\end{theorem}

\begin{proposition}
 {\it For every} $n, m\in \mathbb{N}^{+}$ {\it we have} $R(n, m) \subseteq R(n+1, m)$ .
\end{proposition}
\begin{proof}Pick some $k \in R(n,\ m)$ and let $P \in \mathcal{P}(n, m)$ be such that $|V(P)| = k$. Write $P =p_{1}p_{2}\ldots p_{2n}$. Pick any letter - say $A$, with complement $\bar{A}$ - from the complementary alphabet. We will construct $P'\in \mathcal{P}(n+1,\ m)$ from $P$ by inserting the string $A\bar{A}$ somewhere in $P$, so that $|V(P')| =k$. To achieve this, we want to make sure our two new letters are adjacent and have to get matched with each other in every noncrossing matching. We note that in any noncrossing matching, if $(p_{i}p_{j})$ is an edge with $i<j$, then the number of {\it A}s and number of $\bar{A}$s among $\{p_{i+1}\ldots p_{j-1}\}$ is equal. 

Denote the number of $A$s (and hence also the number of $\bar{A}$s) in $P$ by $a$. Consider a circular speed track which is $a$ miles long, and put $a$ gas stations around it in clockwise order, so that the distance between the $i$-th and $i+1$-th gas station (in miles) is equal to the number of $\bar{A}$s between the $i$-th and $i+1$-th $A$ in $P$. Note that the distance between the $a$-th and first gas station (in miles) is equal to the number of $\bar{A}$s occurring after the last $A$ or before the first $A$ of $P$. Assume each gas station has one liter of gasoline, and our car uses one liter per mile. According to Theorem~\ref{thm:gem} there exists a gas station where we can start with our car with an empty tank and go around the track clockwise without ever running out of gas. 

Say this gas station corresponded to $p_i$. Set $P':=p_1p_2\ldots p_{i-1}\bar{A}Ap_ip_{i+1}\ldots p_{2n}$. By the choice of $i$, in any perfect noncrossing matching our two new added letters have to be connected to each other, hence $|V(P')| = |V(P)| =k$.
\end{proof}

\textbf{Note:} This manuscript was originally written, but not published, in 2015. The paper~\cite{black} it is referencing has since been updated and published, see~\cite{black2}.

\end{document}